\documentclass[11pt]{amsart}
\usepackage{amssymb, amsmath, latexsym, amsthm, amscd, fancyhdr, graphicx,color,tikz-cd}
\usepackage{mathabx,epsfig}
\usepackage{hyperref}
\usepackage{palatino}
\usepackage{euler,color}
\usepackage[margin=1.25in]{geometry}

\usepackage{t1enc}

\numberwithin{equation}{section}

\def\R{\mathbb R}
\def\N{\mathbb N}
\def\C{\mathbb C}
\def\Z{\mathbb Z}
\def\Q{\mathbb Q}
\def\T{\mathbb T}

\def\OO{\mathcal O}
\def\O{\mathcal O}

\def\CC{\mathcal C}
\def\={\equiv}
\def\<{\langle}
\def\>{\rangle}

\def\Cl{\operatorname {\CC}\!\ell}
\def\clk{\Cl_K}

\def\rank{\operatorname{rank}}

\def\ov{\overline}
\def\inv{^{-1}}
\def\ox{\ok^\times}
\def\nx{\N^\times}
\def\ok{\OO_{\! K}}
\def\oxox{\ok\rtimes \ok^\times}
\def\uk{\ok^*}
\def\uf{\OO_{\!F}^*}

\def\hatok{\hat{\OO}_{\! K}}
\def\lsp{\operatorname{span}}
\def\id{\operatorname {id}}

\def\hatok{\widehat{\OO}_{\! K}}
\def\lsp{\operatorname{span}}

\def\Ind{\operatorname{Ind}}
\def\sld{SL_d}
\newcommand{\restr}[1]{\! \!\upharpoonright _{#1}}
\newtheorem{theorem}{Theorem}[section]

\newtheorem{proposition}[theorem]{Proposition}
\newtheorem{corollary}[theorem]{Corollary}
\newtheorem{lemma}[theorem]{Lemma}

\theoremstyle{remark}
\newtheorem{remark}[theorem]{Remark}
\newtheorem{definition}[theorem]{Definition}

\newcommand{\thmref}[1]{Theorem~\ref{#1}}

\newcommand{\proref}[1]{Proposition~\ref{#1}}
\newcommand{\lemref}[1]{Lemma~\ref{#1}}

\def\acts{\mathrel{\reflectbox{$\righttoleftarrow$}}}
\title[Phase transitions of C*-algebras and Furstenberg conjecture]{Phase transitions on C*-algebras 
arising from number fields and the generalized Furstenberg conjecture}

\author{Marcelo Laca} 
\address[M. Laca]{Department of Mathematics and Statistics, University of Victoria, Canada}
\author{Jacqueline M. Warren}
\address[J. M. Warren]{Department of Mathematics, University of California, San Diego}

\date{January 23, 2018}
\begin{document}
\begin{abstract}
In recent work,  Cuntz, Deninger and Laca have studied the Toeplitz type
C*-algebra associated to the affine monoid of algebraic integers in a number field, under a time evolution determined by the absolute norm.
The KMS equilibrium states of their system are parametrized by traces on the C*-algebras 
of the semidirect products $J_\gamma \rtimes \uk$ resulting from the multiplicative action of the
units $\uk$ on integral ideals $J_\gamma$ representing each ideal class $\gamma \in \clk$. At each fixed inverse temperature $\beta >2$,  the extremal equilibrium states
correspond to extremal traces of $C^*(J_\gamma \rtimes \uk)$. 
Here we undertake the study of these traces using the transposed action of $\uk$ on the duals $\hat{J}_\gamma$ of the ideals and the recent characterization of traces on transformation group C*-algebras due to Neshveyev.
 We show that the extremal traces of $C^*(J_\gamma \rtimes \uk)$ are parametrized by pairs consisting of 
an ergodic invariant measure for the action of $\uk$ on $\hat{J}_\gamma$ together with a
 character of the isotropy subgroup associated to the support of this measure.
For every class $\gamma$, the dual group $\hat{J}_\gamma$ is a $d$-torus on which
$\uk$ acts by linear toral automorphisms. Hence, the problem
of classifying all extremal traces is a generalized version of Furstenberg's
celebrated $\times_2$ $\times_3$ conjecture.
 We classify the results for various number fields in terms of ideal class group, degree, and unit rank, and we point along the way the trivial, the intractable, and the conjecturally classifiable cases.
At the topological level, it is possible to characterize the number fields for which infinite $\uk$-invariant sets are dense 
in $\hat{J}_\gamma$, thanks to a theorem of Berend; as an application we give a description of the primitive ideal space of $C^*(J_\gamma \rtimes \uk)$ for those number fields. 
\end{abstract}

\maketitle

\section{Introduction}
 Let $K$ be an algebraic number field and let $\ok$ denote its ring of integers.
 The associated multiplicative monoid $\ox := \ok \setminus \{0\}$ of nonzero integers  acts by injective endomorphisms on the additive group of $\ok$ and gives rise to the semi-direct product $\oxox$, the affine monoid (or `$b+ax$ monoid') of algebraic integers in $K$. 
 
Let $\{\xi_{(x,w)}: (x,w) \in \oxox\}$  be the standard orthonormal basis of the Hilbert space $\ell^2(\oxox)$. The left regular representation $L$ of $ \oxox$ by isometries on $\ell^2(\oxox)$ is determined by $L_{(b,a)} \xi_{(x,w)} = \xi_{(b+ax,aw)}$. In \cite{CDL}, Cuntz, Deninger and Laca studied the Toeplitz-like  C*-algebra 
$\mathfrak{T} [\ok] := C^*(L_{(b,a)}: (b,a) \in \oxox)$ generated by this representation
and analyzed the equilibrium states of the natural time evolution $\sigma$ on $\mathfrak{T} [\ok]$ determined by the absolute norm $N_a := | \ox/(a)|$ via
\[
\sigma _t (L_{(b,a)}) = N_a^{it} L_{(b,a)} \qquad a\in \ox, \ \  t\in \R.
\]

One of the main results of \cite{CDL} is  a characterization of the simplex of KMS equilibrium states of this dynamical system at each inverse temperature $\beta \in (0,\infty]$. 
Here we will be interested in the low-temperature range of that classification. To describe the result briefly, let $\uk$ be the group of units, that is, the elements of $\ox$ whose inverses are also integers, and recall that by a celebrated theorem of Dirichlet, $\uk \cong W_K \times \Z^{r+s-1}$, where $W_K$ (the group of roots of unity in $\uk$) is finite, $r$ is the number of real embeddings of $K$, and $s$ is equal to half the number of complex embeddings of $K$. 
Let $\clk $ be the ideal class group of $K$, which, by definition, is the quotient of the group of all fractional ideals in $K$ modulo the principal ones, and is a finite abelian group. For each ideal class $\gamma \in \clk$ let $J_\gamma \in \gamma$ be an integral ideal representing $\gamma$. 
 By  \cite[Theorem 7.3]{CDL}, for each $\beta > 2$ the KMS$_\beta$ states of $C^*(\oxox)$ are parametrized by
 the tracial states of the  direct sum of group C*-algebras $\bigoplus_{\gamma\in\clk} C^*(J_\gamma\rtimes \uk)$,  where the units act by multiplication on each ideal viewed as an additive group. 
It is intriguing that exactly the same direct sum of group C*-algebras also plays a role in the computation of the $K$-groups of the semigroup C*-algebras of algebraic integers in the work of Cuntz, Echterhoff and Li, see e.g. \cite[Theorem 8.2.1]{CEL}. Considering as well that the group of units and the ideals representing different ideal classes are a measure of the failure of unique factorization into primes in $\ok$, we feel it is of interest to investigate the tracial states of the
 C*-algebras $C^*(J_\gamma\rtimes \uk)$ that arise as a natural parametrization of KMS equilibrium states of $C^*(\oxox)$. 

This work is organized as follows. In Section \ref{FromKMS} we review the phase transition from \cite{CDL} and apply a theorem of Neshveyev's  to show in \thmref{thm:nesh} that the extremal KMS states arise from ergodic invariant probability measures and characters of their isotropy subgroups for the actions $\uk \acts \hat J_\gamma$ of units on the duals of integral ideals. 

We begin Section \ref{unitaction} by showing that for imaginary quadratic fields, the orbit space of the action of units is a compact Hausdorff space that parametrizes the ergodic invariant probability measures. All other number fields have infinite groups of units leading to `bad quotients' for which noncommutative geometry provides convenient tools of analysis.  
Units act by toral automorphisms and so the classification of equilibrium states is intrinsically related to the higher-dimensional, higher-rank version of the question, first asked by H. Furstenberg, of whether Lebesgue measure is the only nonatomic ergodic invariant measure for the pair of transformations $\times 2$ and $\times 3$ on $\R/\Z$. Once in this framework, it is evident from work of Sigmund  \cite{Sig} and of Marcus \cite{Mar}
on partially hyperbolic toral automorphisms and from the properties of the Poulsen simplex \cite{LOS}, that for fields whose unit rank is $1$, which include real quadratic fields, 
there is an abundance of ergodic measures, \proref{poulsen},
and hence of extremal equilibrium states, see also \cite{katz}. 
We also show in this section that there is solidarity among integral ideals  with respect to the ergodicity properties of the actions of units, \proref{oneidealsuffices}. 

In Section \ref{berendsection}, we look at the  topological version of the problem and 
we identify the number fields for which \cite[Theorem 2.1]{B} can be used to give a complete description of the invariant closed sets.  In \thmref{conjecturalclassification} we summarize the consequences, for extremal equilibrium at low temperature, of the current knowledge on the generalized Furstenberg conjecture. For fields of unit rank at least $2$ that are not complex multiplication fields, i.e. that have no proper subfields of the same unit rank, we show that if there is an extremal KMS state that does not arise from a finite orbit or from Lebesgue measure, then it must arise from a zero-entropy, nonatomic ergodic invariant measure; it is not known whether such a measure exists. For complex multiplication fields of unit rank at least $2$, on the other hand, it is known that there are other measures, arising from invariant subtori.
 As a byproduct, we also provide in \proref{ZWclaim} a proof of an interesting fact stated in \cite{ZW}, namely the units acting on algebraic integers are generic among toral automorphism groups that have Berend's ID property.  

We conclude our analysis in Section \ref{prim} by computing the topology of the quasi-orbit space of the action 
$\uk\acts \hatok$ for number fields satisfying Berend's conditions. As an application we also obtain an explicit description of the primitive ideal space of the C*-algebra $C(\ok \rtimes \uk)$, \thmref{primhomeom}.
For the most part, sections \ref{unitaction} and \ref{berendsection} do not depend on  
operator algebra considerations other than for the motivation and the application, which are discussed in sections \ref{FromKMS} and \ref{prim}. 

\noindent{\sl Acknowledgments:} This research started as an Undergraduate Student Research Award project 
and the authors acknowledge the support from the Natural Sciences and Engineering Research Council of Canada. We would like to thank Martha {\L}\k{a}cka for pointing us to \cite{LOS}, and we also especially thank Anthony Quas for bringing Z. Wang's work \cite{ZW} to our attention, and for many helpful comments, especially those leading to Lemmas \ref{partition} and \ref{Anthony'sLemma1}.

\section{From KMS states to invariant measures and isotropy}\label{FromKMS}
Our approach to describing the tracial states of the C*-algebras $\bigoplus_{\gamma\in\clk} C^*(J_\gamma\rtimes \uk)$
is shaped by the following three observations. First, the tracial states of a group C*-algebra form a Choquet simplex \cite{thoma}, so it suffices to focus our attention on the {\em extremal traces}. Second, there is a canonical isomorphism $C^*(J\rtimes \uk) \cong C^*(J)\rtimes \uk $, which we may combine with the Gelfand transform for $C^*(J)$, thus obtaining an isomorphism of $C^*(J\rtimes \uk)$ to the transformation group C*-algebra $C(\hat{J})\rtimes \uk$, associated to the transposed action of 
 $\uk$ on  the continuous complex-valued functions on the compact dual group $\hat{J}$. 
 Specifically, the action of $\uk$ on $\hat{J}$ is  determined by  
 \begin{equation}\label{actiononjhat}
(u \cdot \chi )(j):= \chi( u j), \qquad u \in \uk,  \ \  \chi \in \hat{J}, \ \  j\in J,
\end{equation}
 or by $\langle j, u \cdot \chi \rangle = \langle u j, \chi\rangle$, if we use
 $\langle \ , \ \rangle$ to denote the duality pairing of $J$ and $\hat J$.
Third, this puts the problem of describing the tracial states squarely in the context of
 Neshveyev's  characterization of traces on crossed products,
so our task is to identify and describe the relevant ingredients of this characterization. 
In brief terms, when \cite[Corollary 2.4]{nes} is interpreted in the present situation, it says that  that for each integral ideal $J$, the  extremal traces on $C(\hat{J})\rtimes \uk$ are parametrized by triples $(H,\chi,\mu)$ in which 
 $H$ is a subgroup of $\uk$, $\chi$ is a character of $H$, and $\mu$ is an ergodic $\uk$-invariant measure on $\hat{J}$ such that the set of points in $\hat{J}$ whose isotropy subgroups for the action of $\uk$ are equal to $H$ has full $\mu$ measure. 
 
Recall that, by definition, an $\uk$-invariant probability measure $\mu$ on $\widehat{J}$ is \emph{ergodic invariant} for the action of $\uk$ if 
$\mu(A) \in\{0,1\}$ for every $\uk$-invariant Borel set $A\subset \hat{J} $. 
Our first simplification is that the action of  $\uk$ on $\hat J$ automatically has $\mu$-almost everywhere constant isotropy 
with respect to each ergodic invariant probability measure $\mu$.

\begin{lemma}\label{automaticisotropy}
Let $K$ be an algebraic number field with ring of integers $\ok$ and group of units $\uk$
and let $J$ be a nonzero ideal in $\ok$. Suppose $\mu$ is an ergodic $\uk$-invariant probability 
measure on $\hat J$. Then there exists a unique subgroup $H_\mu$ of $\uk$ such that 
the isotropy group $(\uk)_\chi:= \{u \in \uk: u\cdot \chi = \chi\}$ is equal to $H_\mu$ for $\mu$-a.a. characters $\chi\in \hat J$.
\end{lemma}
\begin{proof}
For each subgroup $H \leq \uk$, 
let $M_H := \{\chi \in \hat J \mid (\uk)_\chi = H\}$ be the set of characters of $J$ with isotropy equal to $H$. 
Since the isotropy is constant on  orbits, each $M_H$ is $\uk$-invariant,
and clearly the $M_H$ are mutually disjoint. 
By Dirichlet's unit theorem $\uk \cong W_K \times \Z^{r+s-1}$ with $W_K$ finite, and $r$ and $2s$ the number of real and complex embeddings of $K$, respectively.
Thus every subgroup of $\uk$ is generated by at most $|W_K| + (r+s-1)$ generators, 
and hence $\uk$ has only countably many subgroups. Thus $\{M_H: H \leq \uk\}$ is a countable partition of $\hat J$
into  subsets of constant isotropy.  

We claim that each $M_H$ is a Borel measurable set in $\hat J$. To see this, observe:

\begin{align*}
M_H =&\{ \chi \in \hat J: u\cdot \chi = \chi \text{ for all }u \in H\text{ and } u\cdot \chi \neq \chi \text{ for all }u \in \uk\setminus H\}\\
=& \Big(\bigcap\limits_{u \in H} \{\chi \in \hat J \mid \chi\inv(u\cdot\chi)=1\} \Big)\bigcap \Big( \bigcap\limits_{u \in \uk\setminus H} \{\chi \in \hat J \mid \chi\inv(u\cdot \chi)\ne 1\}\Big)
\end{align*}
because $u\cdot \chi = \chi$ iff $\chi\inv(u\cdot\chi)=1$.
Since the map $\chi \mapsto\chi\inv(u\cdot \chi)$ is continuous on $\hat J$,
the sets in the first intersection are closed and those in the second one are open. By above, the intersection is countable, so $M_H$ is Borel-measurable, as desired.

 For every Borel measure $\mu$ on $\hat J$, we have 
 \[
\sum\limits_{H \leq \uk} \mu (M_H) = \mu\Big(\bigcup\limits_{H \leq \uk} M_H \Big)  = 1,
\]
 so at least one $M_H$ has positive measure. If $\mu$ is ergodic $\uk$-invariant, then there exists 
 a unique ${H_\mu \leq \uk}$ such that $\mu(M_{H_\mu}) = 1$ and thus $H_\mu$ is the (constant) isotropy group of $\mu$-a.a points $\chi \in \hat J$.
 \end{proof}

Since each ergodic invariant measure determines an isotropy subgroup, the characterization of extremal traces 
from \cite[Corollary 2.4]{nes} simplifies as follows.

\begin{theorem}\label{thm:nesh} Let $K$ be an algebraic number field with ring of integers 
$\ok$ and group of units $\uk$ and let $J$ be a nonzero ideal in $\ok$.   
Denote the standard generating unitaries of $C^*(J\rtimes \uk)$ by  $\delta_j$  for  $j\in J$ and $\nu_u$ for $u \in \uk$.
Then for each extremal trace $\tau$ on $C^*(J\rtimes \uk)$ there exists a unique  probability measure $\mu_\tau$ on $\hat J$ such that
\begin{equation} \label{mufromtau}
\int_{\hat J} \<j, x\> d\mu_\tau(x) = \tau(\delta_j) \quad \text{ for } j \in J.
\end{equation}
The probability measure $\mu_\tau $ is ergodic $\uk$-invariant, and 
if we denote by $H_{\mu_\tau}$ its associated isotropy subgroup 
from \lemref{automaticisotropy},
then the function $\chi_\tau$ defined  by $\chi_\tau(h):= \tau(\nu_h)$ for $h \in H_{\mu_\tau}$ is a character on $H_{\mu_\tau}$.

 Furthermore, the map $\tau \mapsto (\mu_\tau,\chi_\tau)$ is a bijection of the set of extremal traces of $C^*(J\rtimes \uk)$ onto the set of pairs $(\mu, \chi)$ consisting of an ergodic $\uk$-invariant probability measure $\mu$ on $\hat J$
and a character $\chi \in \widehat H_\mu$. The inverse  map $(\mu,\chi) \mapsto \tau_{(\mu,\chi)}$ is determined by 
\begin{equation} \label{muchi-parameters}
\tau_{(\mu,\chi)}(\delta_j \nu_u) = \begin{cases}\displaystyle\chi(u)\int_{\hat J} \<j, x\>d\mu(x) &\text{ if $u \in H_\mu$}\\ 0 &\text{ otherwise,}\end{cases}
\end{equation}
for $j\in J$ and $u\in \uk$.
\end{theorem}

\begin{proof}  
Recall that equation  \eqref{actiononjhat} gives the continuous action of $\uk$ by automorphisms 
of the compact abelian group $\hat J$ obtained on  transposing the multiplicative action of $\uk$ on $J$.
There is a corresponding action $\alpha$ of $\uk$ by automorphisms of the C*-algebra  $C(\hat J)$ of continuous functions on $\hat{J}$; it is given by $\alpha_u(f) (\chi) = f(u\inv \cdot \chi)$. 

The characterization of traces \cite[Corollary 2.4]{nes} then applies to the crossed product  $C(\hat J) \rtimes_\alpha \uk$ as follows. For a given extremal tracial state $\tau$ of $C^*(J\rtimes \uk)$ there is 
a probability measure $\mu_\tau$ on $\hat J$ that arises, via the Riesz representation theorem, from the restriction 
 of $\tau$ to $C^*(J) \cong C(\hat J)$ and is characterized by its Fourier coefficients 
in equation \eqref{mufromtau}.  By \lemref{automaticisotropy}, there is a subset of $\hat J$ of full $\mu_\tau$ measure on which the isotropy subgroup is automatically constant, and is denoted by $H_{\mu_\tau}$. The unitary elements $\nu_u$ generate a copy of $C^*(\uk)$ inside $C(\hat J) \rtimes_\alpha \uk$ and the restriction of $\tau$ to these generators determines a character $\chi_\tau$ of $H_{\mu_\tau}$ given by $\chi_\tau(u) := \tau(\nu_u)$. See the proof of \cite[Corollary 2.4]{nes} for more details. 
 By \lemref{automaticisotropy}, the condition of almost constant isotropy is automatically satisfied for every ergodic invariant measure on $\hat J$, hence every ergodic invariant measure arises as $\mu_\tau$ for some extremal trace $\tau$. The parameter space for extremal tracial states
  is thus the set of all pairs $(\mu,\chi)$ consisting of an ergodic $\uk$-invariant probability measure $\mu$ on $\hat J$ and a character $\chi$ of the  isotropy subgroup $H_\mu$ of $\mu$. Formula \eqref{muchi-parameters} is a particular case of the formula in \cite[Corollary 2.4]{nes} with $f$ equal to the character function $f(\cdot) = \<j,\cdot\>$ on $\hat J$ associated to $j\in J$.
 Since for a fixed $u \in \uk$ the right hand side of \eqref{muchi-parameters} is a continuous linear functional of the integrand
 and the character functions span a dense subalgebra, this particular case is enough to imply
  \begin{equation}
\tau_{(\mu,\chi)}(f\nu_u) = \begin{cases}\displaystyle\chi(u)\int_{\hat J} f(x)d\mu(x) &\text{ if $u \in H_\mu$}\\ 0 &\text{ otherwise,}\end{cases}
\end{equation}
  for every $f\in C(\hat J)$.
 \end{proof}
 
 \section{The action of units on integral ideals}\label{unitaction}
Combining \cite[Theorem  7.3]{CDL} with \thmref{thm:nesh} above, we see that for $\beta> 2$,
the extremal KMS$_\beta$ equilibrium states of the system $(\mathfrak{T} [\ok], \sigma)$ 
are indexed by pairs $(\mu,\kappa)$ consisting of an ergodic invariant probability measure $\mu$ and a character
$\kappa$ of its isotropy subgroup relative to the action of the unit group $\uk$ on a representative of each ideal class.

If the field $K$ is imaginary quadratic, that is, if $r=0$ and $s=1$, then the group of units is finite, 
consisting exclusively of roots of unity. In this case, 
things are easy enough to describe because the space of $\uk$-orbits in $\hat{J}$ is a compact Hausdorff topological space. 
 \begin{proposition} \label{imaginaryquadratic}Suppose $K$  is an imaginary quadratic number field, let $J \subset \ok$ be an integral ideal
 and write $W_K$ for the group of units. Then the orbit space $W_K \backslash \hat J $ is a compact Hausdorff space 
 and the closed invariant sets in $\hat J$ are indexed by the closed sets in $W_K \backslash \hat J $.
Moreover, the ergodic invariant probability measures on $\hat J$ are the
equiprobability measures on the orbits and correspond to unit point masses on $W_K\backslash \hat J$.
 \end{proposition}
 \begin{proof}
 Since $W_K$ is finite, distinct orbits are separated by disjoint invariant open sets, so the quotient space $W_K \backslash \hat J $ is a compact Hausdorff space. Since $\hat J$ is compact, the quotient map $q:  \hat J \to  W_K \backslash \hat J $ given by  $q(\chi ) := W_K \cdot \chi$ is a closed map by the closed map lemma, and so invariant closed sets in $\hat J$ correspond to closed sets in the quotient. 
 
 For each probability measure $\mu$ on $\hat J$, there is a probability measure $\tilde\mu$ on $W_K \backslash \hat J$   defined by \[\tilde\mu(E) := \mu(q\inv(E))\quad  \text{ for each measurable } E\subseteq W_K \backslash \hat J.\]
This maps the set of $W_K$-invariant probability measures on $\hat J$ onto the set of
 all probability measures on  $W_K \backslash \hat J $. Ergodic invariant measures correspond to unit point masses on $W_K\backslash \hat J$, and their $W_K$-invariant lifts are equiprobability measures on single orbits in $\hat J$. 
 \end{proof}
 As a result we obtain the following characterization of extremal KMS equilibrium states.
\begin{corollary}
Suppose $K$ is an imaginary quadratic algebraic number field and 
let $J_\gamma$ be an integral ideal representing the ideal class $\gamma\in \clk$.
For $\beta>2$, the extremal KMS$_\beta$ states of the system $(\mathfrak T[\ok], \sigma^N)$
are parametrized by the triples $(\gamma, W\cdot \chi, \kappa)$, where $\gamma\in \clk$, $\chi$ is a point in 
$ \hat J_\gamma$,  with orbit $W\cdot \chi$ and $\kappa$ is a character of the isotropy subgroup of $\chi$.
\end{corollary}

Before we discuss invariant measures and isotropy  for fields with  infinite group of units, we need to revisit a few general  facts about the multiplicative action of units on the algebraic integers and, more generally, on the integral ideals. The concise discussion in \cite{ZW}
is particularly convenient for our purposes. As is customary, we let $d = [K:\Q]$ be the {\em degree} of $K$ over $\Q$.
The number $r$ of real embeddings and the number $2s$ of complex embeddings satisfy $r+2s=d$. 
We also let $n = r+s -1$ be the {\em unit rank} of $K$, namely, the free abelian rank of $\uk$ according to Dirichlet's unit theorem.
We shall denote the real embeddings of $K$  by
 $\sigma_j:K \to \R$ for $j = 1, 2, \cdots r$ and the conjugate pairs of complex embeddings of $K$ by
$\sigma_{r +j}, \sigma_{r+s+j} : K \to \C$ for $ j = 1, \cdots, s$.
Thus, there is an isomorphism 
\[
\sigma: K \otimes_\Q \R \to \R^r \times \C^s 
\]
such that 
\[ \sigma (k\otimes x) = (\sigma_1( k) x, \sigma_2(k) x, \cdots, \sigma_r(k) x;\, \sigma_{r+1}(k) x, \cdots, \sigma_{r+s}(k) x ).
\]
The ring of integers $\ok$ is a free $\Z$-module of rank $d$, and thus $\ok\otimes_\Z \R \cong \R^d \cong \R^r \oplus \C^s$.
We temporarily fix an integral basis for $\ok$, which fixes an isomorphism 
$\theta: \ok \to \Z^d$. Then, at the level of $\Z^d$,  the action of each $u \in \uk$ is implemented as left multiplication
by a matrix $A_u \in GL_d(\Z)$. Moreover, once this basis has been fixed,  the usual duality pairing 
$\langle \Z^d, \R^d/\Z^d \rangle $ given by $\langle n, t \rangle = \exp{2\pi i (n \cdot t)} $, with $n\in \Z^d$, $t\in \R^d$ and $n\cdot t = \sum_{j=1}^d n_j t_j$, gives an isomorphism of
$\R^d/\Z^d $ to $\hatok$, in which the character $\chi_t\in \hatok$ corresponding to $t\in \R^d/\Z^d$ is given by 
$\chi_t(x) = \exp{2\pi i (\theta(x) \cdot t)} $ for $x\in \ok$.
Thus, the action of a unit $u\in \uk$ is
\[(u\cdot \chi_t)(x) = \chi_t(u\cdot x) = \exp{2\pi i (A_u \theta(x) \cdot t)} = \exp{2\pi i ( \theta(x) \cdot A_u^T t)}.\]
 This implies that the action $\uk \acts \hatok$ is implemented, at the level of $\R^d/\Z^d$, by the representation 
 $\rho: \uk \to GL_d(\Z)$ defined by $\rho(u) = A_u^{T}$, cf. \cite[Theorem 0.15]{Wal}.

Similar considerations apply to the action of $\uk$ on $\hat J$ for each integral ideal $J \subset \ok$,
giving a representation $\rho_J: \uk \to GL_d(\Z)$.
For ease of reference we state the following fact about this matrix realization $\rho_J$ 
of the action of $\uk$ on $\hat{J}$.
\begin{proposition}\label{diagonalization}
The collection of matrices $\{\rho(u) : u \in \uk\}$ is simultaneously diagonalizable (over $\C$) ,
and for each $u\in \uk$ the eigenvalue list of $\rho(u)$ 
is the list of its archimedean embeddings $\sigma_k(u): k = 1, 2, \cdots r+2s$. 
\end{proposition}
See e.g. the discussion in \cite[Section 2.1]{LW}, and \cite[Section 2.1]{ZW} for the details. Multiplication of complex numbers in each complex embedding is regarded as the action of 2x2 matrices on $\R+i\R \cong \R^2$, and the  2x2 blocks corresponding to complex roots simultaneously diagonalize over $\C^d$.
The self duality of $\R^r\oplus \C^s $ can be chosen to be
compatible with the isomorphism mentioned right after (2.1) in \cite{LW} and with multiplication by units.
 See also \cite[Ch7]{KlausS}. 
 
When the number field $K$ is not imaginary quadratic, then $\uk$ is infinite and so
the analysis of orbits and invariant measures is much more subtle; 
for instance, most orbits are infinite, some are dense, and the orbit space does not have a Hausdorff topology.
We summarize for convenience of reference the known basic general properties in the next proposition.

\begin{proposition}\label{orbitsandisotropy}
Let $K$ be a number field with $\rank(\uk) \geq 1$, and let $J$ be an ideal in $\ok$. 
Then normalized Haar measure on $\hat J$ is ergodic $\uk$-invariant, and for each $\chi \in \hat J$,
\begin{enumerate}
\item  
the orbit $\uk\cdot \chi$ is finite if and only if $\chi$ corresponds to a point with rational coordinates in the identification $\hat J \cong \R^d /\Z^d$; in this case the corresponding isotropy subgroup is a full-rank subgroup of $\uk$;
\item the orbit $\uk\cdot \chi$ is infinite if and only if $\chi $ corresponds to a point with at least one irrational coordinate in  $\R^d /\Z^d$; 
\item the characters $\chi$ corresponding to points $(w_1,w_2,\ldots,w_d) \in \R^d$ such that the numbers $1, w_1, w_2, \ldots w_d$ are rationally independent have trivial isotropy.
\end{enumerate}
\end{proposition}
\begin{proof}
By \proref{diagonalization}, for each $u\in \uk$, the eigenvalues of the matrix $\rho(u)$ encoding the action of $u$ at the level of $\R^d/\Z^d$ are precisely the various embeddings of $u$ in the archimedean completions of $K$. 
Since $\rank(\uk) \geq 1$, there exists a non-torsion element $u \in \uk$, whose eigenvalues are not roots of unity.
Hence normalized Haar measure is ergodic for the action of $\{\rho(u): u\in\uk\}$ by \cite[Corollary 1.10.1]{Wal} 
and the first assertion now follows from \cite[Theorem 5.11]{Wal}.
The isotropy is a full rank subgroup of $\uk$ because $|\uk/(\uk)_x| = |\uk\cdot x| < \infty$. 

Let $w = (w_1, w_2, \cdots, w_d)$ be a point in $\R^d/\Z^d$ such that
$1, w_1, \ldots, w_d$ are rationally independent.
Suppose $w$ is a fixed point for the matrix $\rho(u) \in GL_d(\Z)$ 
acting on $\R^d/\Z^d$. 
 Then $\rho(u) w = w$ (mod $\Z^d$) and hence $(\rho(u) - I)w \in \Z^d$, i.e.
 \[[(\rho(u)-I)w]_i = \sum\limits_{j=1}^d (\rho(u)-I)_{ij}w_j \in \Z\] 
 for all $1 \le i \le d$. Since $(\rho(u)-I)_{ij} \in \Z$ for all $i,j$, the rational independence of $1, w_1, \ldots, w_d$ implies that $\rho(u) = I$, so $u=1$, as desired. 
 \end{proof}

We see next that for the number fields with unit rank 1
there are many more ergodic invariant probability measures on $\hatok$ than just Haar measure 
and measures supported on finite orbits. In fact, a smooth parametrization of these measures and of the corresponding KMS equilibrium states of $(\mathfrak T[\ok], \sigma)$ seems unattainable.

\begin{proposition}\label{poulsen}
Suppose the number field $K$ has unit-rank equal to $1$, namely, $K$ is real quadratic, mixed cubic, or complex quartic.
Then the simplex of ergodic invariant probability measures on $\hatok$ is isomorphic to the Poulsen simplex \cite{LOS}.
\end{proposition}

\begin{proof}
The fundamental unit gives a partially hyperbolic toral automorphism of $\hatok$, for which Haar measure is ergodic invariant.
By \cite{Mar,Sig}, the invariant probability measures of such an automorphism that are supported on finite orbits are dense in the space of all invariant probability measures. This remains true when we include the torsion elements of $\uk$. Since these equiprobabilities supported on finite orbits are obviously ergodic invariant and hence extremal among invariant measures, it follows from \cite[Theorem 2.3]{LOS} that the simplex of invariant probability measures on $\hatok$ is isomorphic to the Poulsen simplex.  
\end{proof}

 For fields with unit rank at least  $2$,  whether normalized Haar measure and equiprobabilities supported on finite orbits are the only ergodic $\uk$-invariant probability measures is a higher-dimensional version of the celebrated Furstenberg conjecture, according to which Lebesgue measure is the only non-atomic probability measure on $\T = \R/\Z$ that is jointly ergodic invariant for the  transformations $\times 2$ and $\times 3$ on $\R$ modulo $\Z$. As stated, this remains open, however, Rudolph and Johnson have proved that if $p$ and $q$ are multiplicatively independent positive integers, then the only probability measure on $\R/\Z$ that is ergodic invariant for
 $\times p$ and $ \times q$ and has non-zero entropy is indeed Lebesgue measure \cite{Rud,Joh}. Number fields always give rise to automorphisms of tori of dimension at least $2$, so, strictly speaking the problem in which we are interested does not contain Furstenberg's original formulation as a particular case. 
Nevertheless, the higher-dimensional problem is also interesting and open as stated in general, and there is significant recent activity on it and on closely related problems \cite{KS,KK, KKS}.
In particular, see \cite{EL1} for a summary of the history and also a positive entropy result for higher dimensional tori along the lines of the Rudolph--Johnson theorem. 
We show next that the toral automorphism groups arising from different integral ideals have a solidarity property with respect to the generalized Furstenberg conjecture.

\begin{proposition} \label{oneidealsuffices}
 If for some integral ideal $J$ in $\ok$ the only ergodic $\uk$-invariant probability 
 measure on $\hat J$ having infinite support is normalized Haar measure, then 
the same is true for every integral ideal in $\ok$.
\end{proposition} 
The proof depends on the following lemmas.

\begin{lemma}\label{partition}
Let $J\subseteq I$ be two integral ideals in $\ok$ and let $r: \hat I \to \hat J$ be the restriction map. Denote by $\lambda_{\hat I}$ normalized Haar measure on $\hat I$.
 For each 
$\gamma \in \hat J$, there exists a neighborhood $N$ of $\gamma$ in $\hat J$ and homeomorphisms
$h_j$ of $N$  into $\hat I$ for $j = 1, 2, \ldots , | I/J |$, with mutually  disjoint images and such that
\begin{enumerate}
\item $\lambda_{\hat I} (h_j(E))  = \lambda_{\hat I} ( h_k(E))$ for every measurable $E\subseteq N$ and  $1\leq j, k \leq | I/J |$;
\item $r\circ h_j = \id_N$;
\item $r\inv ( E) = \bigsqcup_j h_j(E)$ for all $E \subseteq N$, that is, the $h_j$'s form a complete system of local inverses of $r$ on $N$.
\end{enumerate}
\end{lemma}
\begin{proof}
Let $J^\perp:= \{\kappa\in \hat I: \kappa (j) = 1, \forall j\in J\}$ be the kernel of the restriction map $r: \hat I \to \hat J$. Since 
$J^\perp$ is a subgroup of order $| I/J | <\infty$, and since $\hat I$ is Hausdorff, we may choose a collection $\{ A_\kappa: \kappa \in  J^\perp\}$ of mutually disjoint open subsets of $\hat I$ such that $\kappa \in A_\kappa$ for each $\kappa\in  J^\perp$.  
Define  $B_1:= \bigcap_{\kappa \in J^\perp} \kappa\inv A_\kappa$ and for each $\kappa\in J^\perp$ 
let $B_\kappa := \kappa B_1$. Then 
$\{B_\kappa : \kappa \in J^\perp\}$  is a collection of mutually disjoint open sets
such that $\kappa \in B_\kappa$ and  $r(B_\kappa) = r(B_1)$ for every $\kappa \in J^\perp$. 
We claim that the restrictions $r: B_\kappa \to \hat J$ are homeomorphisms onto their image.
Since the $B_\kappa$ are translates of $B_1$ and since $r$ is continuous and open, it suffices to verify that $r$ is injective on  $B_1$. This is easy to see because if 
 $r(\xi_1) = r(\xi_2)$ for two distinct elements $\xi_1,\xi_2$ of $B_1$, then  $\xi_2 = \kappa \xi_1$ for some $\rho \in J^\perp \setminus \{1\}$, and this would contradict $B_1 \cap \kappa B_1 = \emptyset$. This proves the claim. We may then take $N:  = \gamma \, r(B_1)$	 and define $h_\rho := (r |_{B_\rho})\inv$, for which properties (1)-(3) are now easily verified.
\end{proof}

\begin{lemma}\label{Anthony'sLemma1}
Let $X$ be a measurable space and let $T: X \to X$ be measurable. Suppose that $\lambda$ is an ergodic $T$-invariant probability measure on $X$. If $\mu$ is a $T$-invariant probability measure on $X$ such that $\mu \ll \lambda$, then $\mu = \lambda$.
\end{lemma}
\begin{proof}
Fix $f \in L^\infty(\lambda)$ and define $(A_n f)(x) = \frac{1}{n} \sum\limits_{k=0}^{n-1} f(T^{k}x)$. Let $S = \{x \in X: (A_nf)(x) \to \int_X f d\lambda\}$. By the Birkhoff ergodic theorem, we have that $\lambda(S^c)=0$, and so $\mu(S^c)=0$ as well, that is, $(A_nf)(x) \to \int_X f d\lambda$ $\mu$-a.e..  Since $f \in L^\infty(\lambda)$ and $\mu \ll \lambda$, we have that $f \in L^\infty(\mu)$ as well, with $\|f\|_\infty^\mu \le \|f\|_\infty^\lambda$. Observe that $|(A_nf)(x)| \le \|f\|_\infty^\lambda$ for $\mu$-a.e. $x$, and so by the dominated convergence theorem, $\int_X A_n f d\mu \to \int_X\left(\int_X f d\lambda\right) d\mu = \int_X f d\lambda$, with the last equality because $\mu(X)=1$.

Because $\mu$ is $T$-invariant, we have that $\int_X A_n f d\mu = \int_X f d\mu$ for all $n$. Combining this with the above implies that $\int_X f d\lambda = \int_X f d\mu$ for all $f \in L^\infty(\lambda)$. In particular, this holds for the indicator function of each measurable set, and so $\mu = \lambda.$
\end{proof}

\begin{lemma} \label{fromItoJ}
Let $J \subseteq I$ be two integral ideals in $\hatok$ and let $r: \hat I \to \hat J$ be the restriction map. 
If $\mu$ is an ergodic $\ok^*$-invariant probability measure on $\hat I$, then $\tilde{\mu}:= \mu \circ r\inv$ is an ergodic invariant probability measure on $\hat J$. Moreover, the support of $\mu$ is finite if and only if the support of $\tilde \mu$ is finite.
\end{lemma}

\begin{proof}
Assume $\mu$ is  ergodic invariant on $\hat I$ and let $E \subseteq \hat J$ be an $\uk$-invariant 
measurable set. Since $r$ is $\uk$-equivariant, $r\inv(E)$ is also $\uk$-invariant so $\tilde{\mu}(E):=\mu(r\inv(E)) \in \{0,1\}$ because $\mu$ is ergodic invariant. Thus, $\tilde{\mu}$ is also ergodic invariant. The statement about the support follows immediately because $r$ has finite fibers. 
\end{proof}

\begin{lemma}\label{liftingF}
Suppose $J \subseteq I$ are integral ideals in $\ok$, and let $\lambda_{\hat J}, \lambda_{\hat I}$ be normalized Haar measures on $\hat J, \hat I$, respectively. If the only ergodic $\ok^*$-invariant probability measure on $\hat J$
with infinite support is $\lambda_{\hat J}$, then the only ergodic $\ok^*$-invariant probability measure on $\hat I$ 
with infinite support is $\lambda_{\hat I}$.
\end{lemma}
\begin{proof}
Let $\mu$ be an ergodic $\ok^*$-invariant probability measure on $\hat I$ with infinite support.  By \lemref{fromItoJ}, $\mu \circ r\inv$ is an ergodic $\ok^*$-invariant probability measure on $\hat J$ with infinite support, and so by assumption must equal $\lambda_{\hat J}$. In particular, $\lambda_{\hat I} \circ r\inv = \lambda_{\hat J}$.

Since $\hat J$ is compact, the open cover $\{N_\gamma: \gamma \in \hat J\}$ given by the sets constructed in \lemref{partition}
has a finite subcover, that is, there exist $\gamma_1, \ldots, \gamma_n \in \hat J$ so that $\hat J = \bigcup\limits_{k=1}^n N_{\gamma_k}$, where  $N_{\gamma_k}$ is a neighborhood of  $\gamma_k \in \hat J$ satisfying the conditions stated in \lemref{partition}, with corresponding maps $h_{\gamma_k}^{(j)}$,  for $1 \le j \le |I/J|$ and  $1 \le k \le n$. 

We will first show that if $B \subseteq \hat I$ is such that $r|_B$ is a homeomorphism  with $r(B) \subseteq N_{\gamma_k}$ for some $k$, and if $\lambda_{\hat I}(B) = 0$, then  $\mu(B) = 0$. Suppose $B$ is such a set and  $\lambda_{\hat I}(B)=0$. By part (3) of \lemref{partition}, $r\inv(r(B)) = \bigsqcup\limits_{j=1}^{|I/J|} h_{\gamma_k}^{(j)}(r(B))$, so $r\inv(r(B))$ is a disjoint union of $|I/J|$ sets, all having the same measure under $\lambda_{\hat I}$. Moreover, there exists some $1 \le j \le |I/J|$ such that $h_{\gamma_k}^{(j)}(r(B)) = B$, because the $h_{\gamma_k}^{(j)}$'s form a complete set of local inverses for $r$, and $r$ is injective on $B$. Putting these together yields
 \[\lambda_{\hat I}(r\inv(r(B))) = |I/J| \lambda_{\hat I}(h_{\gamma_k}^{(j)}(r(B))) = |I/J|\lambda_{\hat I}(B) =0.\] 
 Since $\mu \circ r\inv = \lambda_{\hat J} = \lambda_{\hat I} \circ r\inv$, this implies that $\mu(r\inv(r(B))) = 0$ as well, and since $B \subseteq r\inv(r(B))$, we have that $\mu(B) = 0$.

Now, since $r: \hat I \to \hat J$ is a covering map, for each $\chi \in \hat I$, there exists an open neighbourhood $U_\chi$ of $\chi$ such that $r|_{U_\chi}$ is a homeomorphism. Let $1 \le k \le n$ be such that $r(\chi) \in N_{\gamma_k}$, and let $W_\chi:=U_\chi \cap r\inv(N_{\gamma_k})$. This forms another open cover of $\hat I$, and so by compactness of $\hat I$, there exists a finite subcover $W_1, \ldots, W_m$.

Finally, let $A \subseteq \hat I$ be such that $\lambda_{\hat I}(A)=0$. Then $A \cap W_i$ is a set on which $r$ acts as a homeomorphism, and there exists $1 \le k \le n$ such that $r(A \cap W_i) \subseteq N_{\gamma_k}$. Thus, by the above, we conclude $\mu(A \cap W_i) = 0$ for all $1 \le i \le m$. Since these sets cover $A$, we have that $\mu(A)=0$, and hence $\mu \ll \lambda_{\hat I}$, as desired. By \lemref{Anthony'sLemma1} it follows that $\mu = \lambda_{\hat I}$.
\end{proof}

\begin{proof}[Proof of \proref{oneidealsuffices}:]
Suppose $J$ is an integral ideal such that the only $\uk$-invariant probability measure on $\hat J$
having an infinite orbit is normalized Haar measure. By \lemref{liftingF} applied to the inclusion $J\subset \ok$, the only ergodic $\ok^*$-invariant probability measure on $\hatok$ with infinite support is normalized Haar measure.
 
 Suppose now $I \subseteq \ok$ is an arbitrary integral ideal.
Since the ideal class group is finite, a power of $I$ is principal and thus we may choose $q \in \O^\times_K$ such that $q \ok \subseteq I$. The action of $\ok^*$ on $\hatok$ is conjugate to the action of $\ok^*$ on $\widehat{q\ok}$, and so the only ergodic $\ok^*$-invariant probability measure on $\widehat{q\ok}$ with infinite support is normalized Haar measure. Thus, by  \lemref{liftingF} again with $\widehat{q\ok} \subseteq \hat I$, we conclude that the only ergodic $\ok^*$-invariant probability measure on $\hat I$ is $\lambda_{\hat I}$.
\end{proof}
 
In order to understand the situation for number fields with unit rank higher than $1$,
we review in the next section the topological version of the problem
of ergodic invariant measures, namely, the classification of closed invariant sets.

\section{Berend's theorem and number fields}\label{berendsection}
An elegant generalization to  higher-dimensional tori of Furstenberg's characterization \cite[Theorem IV.1]{F}  
of closed invariant sets for semigroups of  transformations of the circle was obtained by Berend  \cite[Theorem 2.1]{B}.
The fundamental question investigated by Berend is whether an infinite invariant set is necessarily dense, and his original formulation is for semigroups of endomorphisms of a torus. Here we are interested in the specific situation arising from an algebraic number field $K$ in which the units $\uk$ act by automorphisms on $\hat J$ for integral ideals $J \subseteq \ok$ representing each ideal class, so we paraphrase Berend's Property ID  for the special case of a group action on a compact space.
\begin{definition} (cf. \cite[Definition 2.1]{B}.)
Let $G$ be a group acting on a compact space $X$ by homeomorphisms. 
We say that the action $G\acts X$ {\em satisfies the ID property}, or that it
has the {\em infinite invariant dense property},  
if the only closed infinite $G$-invariant subset of $X$ is $X$ itself.
\end{definition}
The first observation is a topological version of the measure-theoretic solidarity proved in \proref{oneidealsuffices};
namely, if $K$ is a given number field, then the action  $\uk \acts \hat J$ has the ID property either for all integral ideals  $J$,  or for none.  
 
\begin{proposition}\label{IDforallornone}
Suppose $K$ is an algebraic number field, and let $J$ be an ideal in $\ok$. Then the action of $\uk$ on $\hat J$ is ID if and only if the action of $\uk$ on $\hat{\ok}$ is ID.
\end{proposition}
\begin{proof}
Suppose first that $J_1 \subseteq J_2$ are ideals in $\ok$ and assume that the action of $\uk$ on $\hat{J}_2$ is ID.  The restriction map $r: \hat J_2 \to \hat J_1$ is $\uk$-equivariant, continuous, surjective, and has finite fibers. Thus, if $E$ were a closed, proper, infinite $\uk$-invariant subset of $\hat{J}_1$, then $r\inv(E)$ would be a closed, proper, infinite $\uk$-invariant subset of $\hat{J}_2$, contradicting the assumption that  the action of $\uk$  on $\hat{J}_2$ is ID. So  no such set $E$ exists, proving that the action of $\uk$ on $\hat{J}_1$ is also ID.

In particular, if the action of $\uk$  on $\hat{\ok}$ is ID, then the action on $\hat{J}$ is also ID for every integral ideal $J\subset \ok$. For the converse, recall that, as in the proof of \proref{oneidealsuffices}, there exists an integer $q \in \ok^\times$ such that $q\ok \subseteq J$, so we may apply the preceding paragraph to this inclusion. Since the action of $\uk$ on $\widehat{q\ok}$ is conjugate to that on $\hatok$, this  completes the proof.
\end{proof}

In order to decide for which number fields the action
of units on the integral ideals is ID, we need to recast Berend's necessary and sufficient conditions
in terms of properties of the number field.
Recall that, by definition, a number field is called a {\em complex multiplication (or CM)  field} if it is a totally imaginary quadratic extension of a totally real subfield. These fields were studied by Remak \cite{rem}, who observed that they are exactly the fields that have a {\em unit defect}, in the sense that they contain a proper subfield $L$ 
with the same unit rank. 

\begin{theorem}\label{berend4units}
 Let $K$ be an algebraic number field and let $J$ be an ideal in $\ok$. 
 The action of $\uk$ on $\hat J$ is ID if and only if $K$ is not a CM field and $\rank \uk \geq 2$.
\end{theorem}

For the proof  we shall need a few
number theoretic facts. We believe these are known but we include the
relatively straightforward proofs below for the convenience of the reader.

\begin{lemma}\label{99percent}
Suppose $\mathcal F$ is a finite family of subgroups of $\Z^d$ such that $\rank(F) < d$ for every $F\in \mathcal F$. 
Then there exists $m\in \Z^d$ such that $m+F$ is nontorsion in $\Z^d/F$ for every $F \in \mathcal F$.
\end{lemma}
\begin{proof}
Recall that for each subgroup $F$ there exists a basis 
$\{n^F_j\}_{j= 1, 2, \ldots , d}$ of $\Z^d$ and integers $a_1, a_2, \ldots, a_{\rank(F)}$ 
such that 
\[
F = \textstyle\big\{ \sum_{i=1}^{\rank(F)} k_i n^F_i:\  k_i \in a_i\Z, \ 1\leq i \leq \rank(F)\big\}.
\]

The associated vector subspaces $S_F := \lsp_\R\{n^F_1, \ldots, n^F_{\rank(F)}\} $ of $\R^d$ are proper and closed
so $\R^d \setminus \cup_F S_F$ is a nonempty open set, see e.g. \cite[Theorem 1.2]{rom}. Let $r$ be a  point in 
$\R^d \setminus \cup_F S_F$ with rational coordinates. If $k$ denotes the l.c.m. of all the denominators of the coordinates of $r$, then $m := kr \in  \Z^d$ and its image $m+F \in \Z^d/F$ is of infinite order for every $F$ because $m \notin S_F$.
\end{proof}

\begin{proposition}\label{unitdefect}
 Let $K$ be an algebraic number field. Then there exists a unit $u\in \uk$ such that
 $K = \Q(u^k)$ for every $k\in \nx$ if and only if $K$ is not a CM field.
\end{proposition}
\begin{proof}Assume first $K$ is not a CM field. Then $\rank \uf < \rank \uk$ for every proper 
subfield $F$ of $K$. Since there are only finitely many proper subfields $F$ of $K$,  \lemref{99percent}
gives a unit $u\in \uk$ with nontorsion image in $\uk/\uf$ for every $F$. Thus $u^k \notin F$  for
every proper subfield $F$ of $K$ and  every $k\in \N$.

Assume now $K$ is a CM field, and let $F$ be a totally real subfield with the same unit rank as $K$ \cite{rem}.
Then the quotient $\uk/\uf$ is finite and there exists a fixed integer $m$ such that 
$u^m \in F$ for every $u\in \uk$.
\end{proof}

\begin{lemma} \label{friday}
Let $k$ be an algebraic number field with $\rank \uk \geq 1$. Then for every embedding 
$\sigma: k \to \C$, there exists $u \in \uk $ such that $|\sigma(u)| > 1$. 
\end{lemma}
\begin{proof}
Assume for contradiction that $\sigma$ is an embedding of $k$ in $\C$ 
such that $\sigma(\uk ) \subseteq \{z \in \C: |z|=1\}$. Let $K = \sigma(k)$ and let $U_K = \sigma(\uk )$.
Then $K \cap \R$ is a real subfield of $K$ with $ U_{K\cap \R} = \{\pm1\}$, so $K \cap \R= \Q$.
Also $K \cap \R$ is the maximal real subfield of $K$, and since  
we are assuming $\rank \uk \geq 1$, $K$ cannot be a CM field. 
To see this, suppose that $k$ were CM. Let $\ell \subseteq k$ be a totally real subfield such that $[k:\ell] = 2$. Since $\ell$ is totally real, $\sigma(\ell)\subseteq \R$, and since $K \cap \R= \Q$, it must be that $\sigma(\ell) = \Q$. Then $\ell = \Q$, so $k$ is quadratic imaginary, contradicting $\rank \uk \ge 1$. 

By \proref{unitdefect}, there exists $u \in U_K$  such that $K = \Q(u)$. Since $|u|=1$, we have that $\ov K =\Q(\ov u)  = \Q(u\inv)=  \Q(u) = K$, so $K$ is closed under complex conjugation. Write $u = a + ib$. Then $u + \ov u = 2a \in K \cap \R= \Q$, so $a \in \Q$. Thus, $K = \Q(u) = \Q(ib).$ Since $|u|=1$, $a^2+b^2=1$, and so we have that \[\Q(ib) \cong \Q(\sqrt{-b^2}) \cong \Q(\sqrt{a^2-1}) \cong \Q\left(\sqrt{\frac{m^2-n^2}{n^2}}\right) \cong \Q(\sqrt{m^2-n^2}),\] where $a = m/n \in \Q$. 

Thus, $K$ is a quadratic field. But it cannot be quadratic imaginary because $\rank U_K \ge 1$, and it cannot be quadratic real because all the units lie on the unit circle. This proves there can be no such embedding. 
\end{proof}

\begin{proof}
[Proof of \thmref{berend4units}]
By \proref{IDforallornone}, it suffices to prove the case $J = \ok$.
Let $d = [K:\Q]$ and recall that $\hatok \cong \T^d$. 
All we need to do is verify that 
Berend's necessary and sufficient conditions for ID \cite[Theorem 2.1]{B}, when interpreted
for the automorphic action of $\uk$ on $\hatok$,
characterize non-CM fields of unit rank $2$ or higher. Since the action of $\uk$ by linear toral automorphisms $\rho(u)$ with $u\in \uk$ is faithful by \cite[p. 729]{KKS}, Berend's conditions are:
 \begin{enumerate}
  \item (totally irreducible) 
 there exists a unit $u$ such that the characteristic polynomial of $\rho(u^n)$ is irreducible for all $n\in \N$;
 \item (quasi-hyperbolic)
 for every common eigenvector of $\{\rho(u): u\in \uk\}$, there is a unit $u\in \uk$ such that the corresponding eigenvalue 
 of $\rho(u)$ is outside the unit disc; and
 \item (not virtually cyclic) 
there exist units $u,v\in \uk$ such that if $m,n \in \N$ satisfy $\rho(u^m) = \rho(v^n)$, then   $m = n = 0$.
\end{enumerate}

Suppose first that the action of $\uk$ on $\ok$ is ID. By \cite[Theorem 2.1]{B} conditions (1) and (3) above hold,
i.e. the action of $\uk$ on $\ok$ is totally irreducible and not virtually cyclic. 
By \proref{unitdefect}, $K$ is not a CM field and since $\rho:\uk \to GL_d(\Z)$ is faithful,  
(3) is a restatement of $\rank\uk \geq 2$.

 Suppose now that $K$ is not CM and has unit-rank at least $2$.
 By  \proref{unitdefect}, there exists $u \in \uk$ such that $\Q(u^n) = K$ for every $n \in \N$. Hence the 
 minimal polynomial of $\rho(u^n)$ has degree $d$, 
and so it coincides with the characteristic polynomial. This proves that condition (1) holds, 
 i.e. the action of $\rho(u)$ is totally irreducible. We have already observed that condition 
 (3) holds iff the unit rank of $K$ is at least $2$, so it remains to see that the hyperbolicity condition (2) holds too.
In the simultaneous diagonalization of the matrix group $\rho(\uk)$, the diagonal entries 
 of $\rho(u)$ are the embeddings of $u$ into $\R$ or $\C$, see e.g. \cite[p.729]{KKS}. Then condition (2) follows from \lemref{friday}.
\end{proof}
\begin{remark}
Notice that for units acting on algebraic integers, Berend's hyperbolicity condition (2) 
is automatically implied by the rank condition (3).
\end{remark}

\begin{remark}
Since the matrices representing the actions of $\uk$ on $\hat J$ and on $\hat \ok$ are conjugate over $\Q$,  \proref{IDforallornone} can be derived from the implication  (1)$\implies$(3) in \cite[Proposition 2.1]{KKS}. We may also see that the matrices implementing the action on $\hat J$ and on $\hat \ok$ have the same sets of characteristic polynomials, so the questions of expansive eigenvalues (condition (2)) and of total irreducibility are equivalent for the two actions. The third condition is independent of whether we look at $\hat J$ or $\hat \ok$, so this yields yet another proof of Proposition \ref{IDforallornone}.
\end{remark}

 By \thmref{berend4units}, for each non-CM algebraic number field $K$ with 
unit rank at least $2$, the  action $\uk$ on $\hatok$, transposed as 
$\{\rho(u): u\in \uk\}$ acting on $\R^d/\Z^d$,
 is an example of an abelian toral automorphism group for which one may hope
 to prove that normalized Haar measure is the only ergodic invariant probability measure with infinite support. So it
 is natural to ask which groups of toral automorphisms arise this way.
A striking observation of Z. Wang \cite[Theorem 2.12]{ZW}, see also \cite[Proposition 2.2]{LW},
states that every finitely generated abelian group of  automorphisms of $\T^d$ 
that contains a totally irreducible element
and whose rank is maximal and greater than or equal to $2$ 
arises, up to conjugacy, from a finite index subgroup
of units acting on the integers of a non-CM field of degree $d$ and unit rank at least $2$, cf.
\cite[Condition 1.5]{ZW}.
We wish next to give a proof of the converse, which was also stated in \cite{ZW}.
\begin{proposition} \label{ZWclaim} Suppose $G$ is an abelian subgroup of $\sld(\Z)$
satisfying \cite[Condition 2.8]{ZW}. Specifically, suppose there exist 
\begin{itemize}
\item a non-CM number field $K$ of degree $d$ and unit rank at least $2$;
\item an embedding $\phi:G \to \uk$ of $G$ into a finite index subgroup of $\uk$;
\item a co-compact lattice $\Gamma$ in $K\subset K\otimes_\Q \R\cong \R^d$ invariant under multiplication by $\phi(G)$; and
\item a linear isomorphism $\psi: \R^d \to K\otimes_\Q \R\cong \R^d$ mapping $\Z^d$ onto $\Gamma$
that intertwines 
  the actions $G\acts\R^d$ and $\phi(G) \acts (K\otimes_\Q \R)/\Gamma$.
\end{itemize}
 
  Then  $G $ satisfies \cite[Condition 1.5]{ZW}, namely
  \begin{enumerate}
  \item $\rank(G)\geq 2$;
  \item the action $g\acts \R^d/\Z^d \cong \T^d$ is totally irreducible for some $g\in G$;
  \item $\rank G_1 = \rank G$ for each abelian subgroup $G_1 \subset\sld(\Z)$  containing~$G$.
  \end{enumerate}
\end{proposition}

\begin{proof}
Suppose $K$ is a non-CM algebraic number field of degree $d$ with unit rank at least $2$,
and assume $G$ is a subgroup of $\sld(\Z)$ that satisfies the assumptions with respect to $K$. 
Part (1) of \cite[Condition 1.5]{ZW} is immediate, because $\phi(G)$ is of full rank in $\uk$.

By \proref{unitdefect},
there exists a  unit $u\in \uk$ such that the characteristic polynomial 
of $\rho(u^m)$ is irreducible over $\Q$ for all $m \in \N$. This is equivalent to the
  action of $u^m$ on $\hatok$ being irreducible for all $m \in \N$, see, e.g.  \cite[Proposition 3.1]{KKS}.
Since $\phi(G)$ is of finite index in $\uk$, there exists $N\in \N$ such that $u^N \in \phi(G)$. 
We claim that  $g := \phi\inv(u^N)$ is a totally irreducible element in $G\acts \R^d/\Z^d$.
To see this, it suffices to show that the characteristic polynomial of $g^k$ is irreducible over $\Q$
for every  positive integer $k$. Since the
linear isomorphism $\psi$  intertwines the actions $g^k\acts \T^d$ and $\rho(\phi(g))^k \acts  (K\otimes_\Q\R)/\Gamma$,
the characteristic polynomial of $g^k$ equals the characteristic polynomial of $\rho(\phi(g))^k = \rho(u^{kN})$, which is
irreducible because it coincides with the characteristic polynomial of $u^{kN}$ as an element of the ring $\ok$.
This proves part (2) of Condition 1.5.
 
Suppose now that $G_1$ is 
an abelian subgroup of $\sld(\Z)$ containing $G$ and 
apply the construction from \cite[Proposition 2.13]{ZW} 
(see also \cite{KlausS,EL1})  to the irreducible element $g\in G\subset G_1 \acts \T^d$. 
Up to an automorphism, the resulting number field arising from this construction is $K= \Q(u^N)$, 
and the embedding $\phi_1: G_1 \to \uk$ 
is an extension of $\phi:G\to \uk$. Since $\phi(G)\subset \phi_1(G_1) \subset \uk$
and $\phi(G)$ is of finite index in $\uk$,   
\[\rank (G_1) = \rank\phi_1(G_1) = \rank \uk = \rank \phi(G) = \rank G,\]
and this proves proves part (3) of Condition 1.5.
\end{proof}
As a consequence, we see that the action of  units on the algebraic integers of number fields are generic 
for group actions with Berend's ID property in the following sense, cf. \cite{ZW,LW}.
\begin{corollary}
If $G$ is a finitely generated abelian subgroup of $\sld(\Z)$ of torsion-free rank at least 2 that contains a totally irreducible element and is maximal among abelian subgroups of $\sld(\Z)$ containing $G$, then $G$ is conjugate to
a finite-index toral automorphism subgroup of the action of  $\uk \acts \hatok$ 
for a non-CM algebraic number field $K$ of degree $d$ and unit rank at least $2$.
\end{corollary}

Finally, we summarize what we can say at this point for equilibrium states of C*-algebras associated to number fields with unit rank strictly higher than one. If the generalized Furstenberg conjecture is verified, the following result would complete the classification started in \proref{imaginaryquadratic} and \proref{poulsen}.

Let $K$ be a number field and for each $\gamma \in \clk$ define 
$F_\gamma $ to be the set of all pairs $(\mu, \chi)$ with $\mu$ an equiprobability measure on a finite orbit of the action of $\uk$ in $\hat {J}_\gamma$, and $\chi \in \hat{H}_\mu$, where the $\mu$-a.e.  isotropy group $H_\mu$ is a finite index subgroup of $\uk$. 
Also let $(\lambda_J, 1) $ denote the pair consisting of normalized Haar measure on $\hat{J}$ and the trivial character of its trivial a.e. isotropy group.
Then the map $(\mu,\chi) \mapsto \tau_{\mu,\chi}$ from \thmref{thm:nesh} gives an extremal tracial state of 
$C^*(J_\gamma \rtimes \uk)$ for each pair $(\mu,\chi) \in F_\gamma \sqcup \{(\lambda_{J_\gamma}, 1) \}$.

Recall that the map $\tau \mapsto \varphi_\tau$ from \cite[Theorem 7.3]{CDL} is an affine bijection of all tracial states of  $\bigoplus_{\gamma\in \clk} C^*(J_\gamma\rtimes \uk)$ onto $\mathcal K_\beta$, the simplex of KMS$_\beta$ equilibrium states of the system $(\mathfrak T[\ok], \sigma)$ studied in \cite{CDL}.
\begin{theorem}\label{conjecturalclassification}
Suppose $K$ is an algebraic number field with unit rank at least $2$
and define $\Phi:(\mu,\chi) \mapsto \varphi_{\tau_{\mu,\chi}}$
to be the composition of the maps from \thmref{thm:nesh} and from \cite[Theorem 7.3]{CDL}, assigning a state $\varphi_{\tau_{\mu,\chi}} \in \operatorname{Extr}(\mathcal K_\beta)$ to each pair $(\mu,\chi)$ consisting of an ergodic invariant probability measure $\mu$
in one of the $\hat{J}_\gamma$ and an associated character of the $\mu$-almost constant isotropy $H_\mu$.
Let
 \[
F_K:= \bigsqcup_{\gamma \in \clk} \big(F_\gamma \sqcup \{(\lambda_{J_\gamma}, 1) \}\big)
\]
 be the set of pairs whose measure $\mu$ has finite support or is Haar measure.
Then 
\begin{enumerate}
\item if $K$ is a CM field, then the inclusion $\Phi(F_K) \subset \operatorname{Extr}(\mathcal K_\beta)$ is proper; and
\item if $K$ is not a CM field, and if there exists $ \phi \in \operatorname{Extr}(\mathcal K_\beta) \setminus \Phi(F_K) $
then the measure $\mu$ on $\hat{J}_\gamma$ arising from $\phi$ has zero-entropy and infinite support.
\end{enumerate}
\end{theorem}
\begin{proof}
To prove  assertion (1), recall that when $K$ is a CM field Berend's theorem implies that there are invariant subtori, which have ergodic invariant probability measures on the fibers, cf. \cite{KK,KS}. These measures give rise to tracial states and to KMS states not accounted for in  $\Phi(F_K)$.  Assertion (2) follows from \cite[Theorem 1.1]{EL1}.
\end{proof}

\section{Primitive ideal space}\label{prim}
 The computation of the primitive ideal spaces of  the C*-algebras $C^*(J \rtimes \uk)$ associated to the action of units on integral ideals
 lies within the scope of Williams' characterization in \cite{DW}.  We briefly review the general setting next. Let $G$ be a countable, discrete, abelian group acting continuously on a second countable compact Hausdorff space $X$. We
 define an equivalence relation on $X$ by saying that {\em $x$ and $y$ are equivalent} if $x$ and $y$ have the same orbit closure, i.e. if $\ov{G\cdot x} = \ov{G \cdot y}$. The equivalence class of $x$, denoted by  $[x]$,  is called the {\em quasi-orbit} of $x$, and 
the quotient space, which in general is not Hausdorff, is denoted by $\mathcal{Q}(G \acts X)$ and is called the {\em quasi-orbit space}.  It is important to distinguish the quasi-orbit of a point from the closure of its orbit, as the latter may contain 
other points with strictly smaller orbit closure. 

      Let $\epsilon_x$ denote evaluation at $x\in X$, viewed as a one-dimensional representation of $C(X)$. For each character $\kappa \in \hat G_x$, the pair $(\epsilon_x,\kappa)$  is clearly covariant for the 
 transformation group $(C(X), G_x)$, and the corresponding representation $\epsilon_x\times \kappa$ of $C(X) \rtimes G_x$ gives rise to an induced representation $\Ind_{{G_x}}^G(\epsilon_x\times \kappa)$ of $C(X) \rtimes G$, which is irreducible because $\epsilon_x\times \kappa$ is.
 Since $G$ is abelian and the action is continuous, whenever $x$ and $y$ are in the same quasi-orbit, $[x] = [y]$, the corresponding isotropy subgroups coincide: $G_x = G_y$.  Thus, we may consider an equivalence relation 
 on the product $\mathcal{Q}(G \acts X) \times \hat G$  defined by 
 \[
 ([x],\kappa) \sim ([y],\lambda) \quad \iff \quad [x]=[y] \text{ and } \kappa\restr{G_x} = \lambda\restr{G_x} .
\]
 By \cite[Theorem 5.3]{DW}, the map $(x,\kappa) \mapsto  \ker \Ind_{{G_x}}^G(\epsilon_x\times \kappa)$ induces a homeomorphism of 
 $(\mathcal{Q}(G \acts X) \times \hat G)/_{\!\sim}$ onto the primitive ideal space of the crossed product 
 $C(X)\rtimes G$, see e.g. \cite[Theorem 1.1]{primbc} for more details on this approach. 
 
 We wish to apply the above result to actions $\uk \acts \hat J$ for integral ideals $J$ 
 of non-CM number fields with unit rank at least $2$, as in \thmref{berend4units}.  
  Notice that by  \proref{orbitsandisotropy}  if the orbit $\uk\cdot \chi$ is finite, then it is equal to the quasi-orbit $[\chi]$. 
The first step is to describe the quasi-orbit space for the action of units.
We focus on the case $J = \ok$; ideals representing nontrivial classes behave similarly because of the solidarity established in \proref{IDforallornone}.

\begin{proposition}\label{quasiorbitset} Suppose $K$ is a non-CM algebraic number field with unit rank at least $2$. Then  the quasi-orbit space of the action $\uk \acts \hatok$ is \[
\mathcal{Q}(\uk \acts \hatok) = \{[x]: |\uk\cdot x | < \infty\}\cup \{\omega_\infty\}.
\]
The point $\omega_\infty$ is the unique infinite quasi-orbit  $[\alpha]$ of any $\alpha \in \hatok\cong \R^d/\Z^d$ having at least one irrational coordinate. 
The closed proper subsets are the finite subsets all of whose points are finite (quasi-)orbits. 
 Infinite subsets and subsets that contain the infinite quasi-orbit $\omega_\infty$ are dense in the whole space.
\end{proposition}

\begin{proof}
By \thmref{berend4units}, the closure of each infinite orbit is the whole space.  Thus, the points with infinite orbits
collapse into a single quasi-orbit 
\[
\omega_\infty :=  \{x\in \hatok: |\uk \cdot x| =\infty\} = \{x\in \hatok: \overline{\uk \cdot x} = \hatok\}. 
\]
That this is the set of points with at least one irrational coordinate is immediate from \cite[Theorem 5.11]{Wal}.
When the orbit of $x$ is finite, it is itself a quasi-orbit, which we
view as a point in $\mathcal{Q}(\uk \acts \hatok)$. In this case $x\in \hatok$ has all rational coordinates. 

To describe the topology, recall that the quotient map $q: \hatok \to \mathcal{Q}(\uk \acts \hatok)$ is surjective, continuous and open by the 
Lemma in page 221 of \cite{PG}, see also the proof of Proposition 2.4 in \cite{primbc}.

Any two different finite quasi-orbits $[x]$ and $[y]$ are finite, mutually disjoint subsets of $\hatok$ and as such can be separated 
by disjoint open sets $V$ and $W$, so that $[x]\subset V$ and $[y] \subset W$. Passing to the quotient space, we have
 $[x]\notin q(W)$ and $[y] \notin q(V)$, so $[x]$ and $[y]$ are $T_1$-separated,
which implies that finite sets of finite quasi-orbits are closed in $\mathcal Q(\uk \acts \hatok)$.

The singleton $\{\omega_\infty\}$ is dense in $\mathcal Q(\uk \acts \hatok)$ because every infinite orbit in $\hatok$ is dense by \thmref{berend4units}.
If $A$ is an infinite subset of $\mathcal Q(\uk \acts \hatok)$ consisting of finite quasi-orbits,
then $\bigcup_{[x]\in A} [x]$  is an infinite invariant set in $\hatok$, hence is dense by 
\thmref{berend4units}. 
This implies that $\omega_\infty$ is in the closure of $A$, and hence $A$ is dense
in $\mathcal Q(\uk \acts \hatok)$.
\end{proof}

\begin{theorem} \label{primhomeom} Let $K$ be a non-CM algebraic number field with unit rank at least 2, and let $G = \uk$. The primitive ideal space of $C(\hatok)\rtimes G$ is homeomorphic to the space \begin{equation}\label{primdef}
\bigsqcup\limits_{[x]} \left(\{[x]\} \times \hat{G}_x \right) \end{equation} in which a net $([x_\iota], \gamma_\iota)$ converges to $([x],\gamma)$ iff $[x_\iota] \to [x]$ in $\mathcal{Q}(G \acts \hatok)$ and $\gamma_\iota|_{G_x} \to \gamma|_{G_x}$ in $\hat{G}_x$. 
Notice that if $[x]$ is a finite quasi-orbit, then the net $\{[x_\iota] \}$ is eventually constant equal to $[x]$, and if $[x] = \omega_\infty$, then the condition $\gamma_\iota|_{G_{\omega_\infty}} \to \gamma|_{G_{\omega_\infty}}$ is trivially true because $G_{\omega_\infty} = \{1\}$.
\end{theorem}
\begin{proof}

Consider the diagram below, where $f$ is the quotient map and the vertical map $g$ is defined by $g([([x],\gamma)]) = ([x], \gamma|_{G_x})$, where $[([x],\gamma)]$ denotes the equivalence class of $([x],\gamma)$ with respect to $\sim$.
\[
  \begin{tikzcd}
   \mathcal{Q}(G\acts \hatok) \times \hat G \arrow{r}{f} \arrow[swap]{dr}{g\circ f} & \mathcal{Q}(G\acts \hatok) \times \hat G/\sim \arrow{d}{g} \\
     & \bigsqcup\limits_{[x]} \left(\{[x]\} \times \hat{G}_x \right)
  \end{tikzcd}
\]
By the fundamental property of the quotient map, see e.g. \cite[Theorem 9.4]{wil},  $g\circ f$ is continuous if and only if  $g$
is continuous. 

It is clear that $g$ is a bijection. We  show next that $g \circ f$ is continuous. Suppose that $([x_\iota], \gamma_\iota)$ is a net in $\mathcal{Q}(G\acts \hatok) \times \hat G$ converging to $([x],\gamma)$. Then $[x_\iota] \to [x]$ in $\mathcal{Q}(G\acts \hatok)$, and $\gamma_\iota \to \gamma$ in $\hat G$. Then clearly also $\gamma_\iota|_{G_x} \to \gamma|_{G_x}$ in $\hat{G}_x$ as well. Hence the net $g\circ f([x_\iota], \gamma_\iota)$ converges to $g\circ f([x],\gamma) = ([x], \gamma|_{G_x})$, as desired.

It remains to show that $g\inv$ is continuous, or equivalently, that $g$ is a closed map. Suppose that $W \subseteq \mathcal{Q}(G\acts \hatok) \times \hat G / \sim$ is closed, and suppose that $([x_\iota], \gamma_\iota)$ is a net in $g(W)$ converging to $([x], \gamma)$.

Consider any net $(\tilde{\gamma_\iota})$ in $\hat{G}$ such that $\tilde{\gamma_\iota}|_{G_x} = \gamma_\iota$. By the compactness of $\hat{G}$, there exists a convergent subnet $\tilde{\gamma}_{\iota_\eta}$ with limit $\tilde{\gamma}$. Then $\tilde{\gamma}_{\iota_\eta}|_{G_x} \to \tilde{\gamma}|_{G_x}$ as well, so $\gamma_{\iota_\eta} \to \tilde{\gamma}|_{G_x}$. Since $\hat{G}_x$ is Hausdorff, limits are unique, and hence $\tilde{\gamma}|_{G_x} = \gamma$.

The net $([x_{\iota_\eta}], \tilde{\gamma}_{\iota_\eta})$ converges to $([x], \tilde{\gamma})$ in $\mathcal{Q}(G \acts \hatok) \times \hat{G}$, and since $f$ is continuous, $f([x_{\iota_\eta}], \tilde{\gamma}_{\iota_\eta}) \to f([x],\tilde{\gamma})$. Moreover, $f([x_{\iota_\eta}], \tilde{\gamma}_{\iota_\eta}) = [([x_{\iota_\eta}], \tilde{\gamma}_{\iota_\eta})] \in W$ because $g$ is injective and $g([([x_{\iota_\eta}],\tilde{\gamma}_{\iota_\eta})]) = ([x_{\iota_\eta}],\gamma_{\iota_\eta}) \in g(W)$ by assumption. Since $W$ is closed, $[([x],\tilde{\gamma})] \in W$, and so its image $([x], \gamma) \in g(W)$, as desired.
\end{proof}

\begin{remark}
Recall that $G  \cong W \times \Z ^d$ with $W$ the roots of unity in $G$,
and that the isotropy subgroup $G_x$ is constant on the quasi-orbit $[x]$ of $x$.
If $[x]$ is finite, then $G_x$ is of full rank in $G$, and thus $G_x  \cong V_x \times \Z^d$, with $V_x \subset W$ the torsion part of $G_x$.  Hence, for every finite quasi-orbit $[x]$, we have  $\hat{G_x} \cong \hat V_{[x]} \times \T^d$.
Notice that $\hat V_{[x]} \cong V_{[x]}$ (noncanonically) because $V_x$ is finite. 
\end{remark}


\begin{thebibliography}{99}

\bibitem{B} D. Berend, {\em Multi-invariant sets on tori}, {Trans. Amer. Math. Soc.} \textbf{280} (1983), 509--532.

\bibitem{CDL} J. Cuntz, C. Deninger and M. Laca, {\em $C^*$-algebras of Toeplitz type associated
    with algebraic number fields}, {Math. Ann.} \textbf{355} (2013), 1383--1423.
    
\bibitem{CEL} J. Cuntz, S. Echterhoff,  and X. Li, {\em On the K-theory of the C*-algebra generated by the left regular representation of  an Ore semigroup,} J. Eur. Math. Soc. {\bf 17}  (2015), 645--687. 

\bibitem{EL1} M. Einsiedler and E. Lindenstrauss, {\em Rigidity properties of $\Z^d$ actions on tori and solenoids},
Electr. Res. Announc. Amer. Math. Soc. {\bf 9} (2003), 99--110.

\bibitem{F} H. Furstenberg, {\em Disjointness in ergodic theory, minimal sets, and a problem in diophantine approximation}, Math. Systems Theory {\bf 1} (1967), 1--49.

\bibitem{PG} P. Green, {\em The local structure of twisted covariance
algebras}, Acta Math. {\bf 140} (1978), 191--250.

\bibitem {Joh} A. Johnson, {\em Measures on the circle invariant under multiplication by a nonlacunary
subsemigroup of the integers}, Israel J. Math. {\bf 77} (1992), 211--240.

\bibitem{KK} A. Katok, B. Kalinin, {\em Invariant measures for actions of higher rank abelian groups}, 
Smooth Ergodic Theory and its Applications (Seattle, WA, 1999), Proc. Sympos. Pure Math., {\bf  69} (2001), 593--637.

\bibitem{KS} A. Katok, R. Spatzier, {\em Invariant measures for higher-rank hyperbolic abelian actions}, 
Ergod. Th. and Dynam. Syst. {\bf 16} no. 4 (1996), 751--778.

\bibitem{KKS} A. Katok, S. Katok, and K. Schmidt, {\em Rigidity of measurable structure for $\Z^d$ actions by automorphisms of a torus,} Comment. Math. Helv. {\bf 77} (2002), 718--745.

\bibitem{katz} Y. Katznelson, {\em Ergodic automorphisms of $T^n$ are Bernoulli shifts},  Israel J. Math. {\bf 10} (1971), 186--195.

\bibitem{primbc}
M. Laca and I. Raeburn, {\em The  ideal structure of the Hecke $C^*$-algebra of Bost and Connes,} Math. Ann. {\bf 318} (2000), 433-451.

\bibitem{LOS} J. Lindenstrauss, G. Olsen, and Y. Sternfeld,  {\em The Poulsen simplex}, Ann. Inst. Fourier {\bf 28} (1978), 91--114.

\bibitem{LW} E. Lindenstrauss and Z. Wang, {\em Topological self-joinings of Cartan actions by toral automorphisms,}
Duke Math. J.  {\bf 161} No. 7 (2012), 1305--1350.

\bibitem{Mar} B. Marcus, {\em A note on  periodic points for ergodic toral automorphisms}, Mh. Math. {\bf 89} (1980), 121--129.

\bibitem{nes} S. Neshveyev, {\em KMS states on the $C^*$-algebras of non-principal groupoids},
    {J. Operator Theory} \textbf{70} (2013), 513--530.

\bibitem{rem} R. Remak, {\em \"Uber algebraische {Z}ahlk\"orper mit schwachem {E}inheitsdefekt}, Compositio Math. {\bf 12} (1954), 35--80.

\bibitem{Rud} D.J. Rudolph, {\em $\times 2$ and $\times 3$ invariant measures and entropy}, Ergod. Th. and Dynam.
Syst. {\bf 10} no. 2 (1990), 395--406.

\bibitem{rom} S. Roman, {\em Advanced Linear Algebra}, 3rd Ed.   Graduate Texts in Mathematics,  2008 Springer.

\bibitem{ST} I.N. Stewart and D.O. Tall, {Algebraic Number Theory}, 1987, Chapman and Hall.

\bibitem{KlausS} K. Schmidt, { Dynamical Systems of Algebraic Origin}, 3rd Ed. Graduate Texts in Mathematics, 2008, Springer.

\bibitem{Sig} K. Sigmund, {\em Generic properties of invariant measures for Axiom A diffeomorphisms}, Invent. math. {\bf 11} (1970), 99-109.

\bibitem{thoma} E. Thoma, {\em  \"{U}ber unit\"are {D}arstellungen abz\"ahlbarer, diskreter {G}ruppen},
{Math. Ann.} {\bf 153} (1964), 111--138.

\bibitem{Wal} P. Walters, An Introduction to Ergodic Theory,  Graduate Texts in Mathematics no. 79, 1982, Springer.

\bibitem{ZW} Z. Wang, {\em Quantitative Density under Higher Rank Abelian Algebraic Toral Actions,}
Int. Math. Res. Not. No. 16  (2011), 3744--3821.

\bibitem{wil} S. Willard, {General Topology},  1970, Addison-Wesley, .

\bibitem{DW} D.P. Williams, {\em The topology on the primitive ideal
space of transformation group $C^*$-algebras and C.C.R. transformation
group $C^*$-algebras}, Trans. Amer. Math. Soc. {\bf 226} (1981), 335--359.

\end{thebibliography}
\end{document}